\documentclass[11pt,reqno]{amsart}
\usepackage{amssymb,mathrsfs,graphicx,float,subfig}
\usepackage{ifthen}
\usepackage{colortbl}
\usepackage{cases,appendix}
\definecolor{black}{rgb}{0.0, 0.0, 0.0}
\definecolor{red}{rgb}{1.0, 0.5, 0.5}

\provideboolean{shownotes} 
\setboolean{shownotes}{true} 
\newcommand{\margnote}[1]{
\ifthenelse{\boolean{shownotes}}%
{\marginpar{\raggedright\tiny\texttt{#1}}}%
{}%
}
\newcommand{\hole}[1]{
\ifthenelse{\boolean{shownotes}}%
{\begin{center} \fbox{ \rule {.25cm}{0cm} \rule[-.1cm]{0cm}{.4cm}
\parbox{.85\textwidth}{\begin{center} \texttt{#1}\end{center}} \rule
{.25cm}{0cm}}\end{center}} {} }

\title[Pressureless Euler-Poisson system]{On the pressureless damped Euler-Poisson equations with non-local forces: Critical thresholds and large-time behavior}

\author[Carrillo]{Jos\'{e} A. Carrillo}
\address[Jos\'{e} A. Carrillo]{\newline Department of Mathematics
    \newline Imperial College London, London SW7 2AZ, United Kingdom}
\email{carrillo@imperial.ac.uk}

\author[Choi]{Young-Pil Choi}
\address[Young-Pil Choi]{\newline Fakult\"at f\"ur Mathematik
    \newline  Technische Universit\"at M\"unchen, Boltzmannstra{\ss}e 3, 85748, Garching bei M\"unchen, Germany}
\email{ychoi@ma.tum.de}

\author[Zatorska]{Ewelina Zatorska}
\address[Ewelina Zatorska]{\newline Department of Mathematics
    \newline Imperial College London, London SW7 2AZ, United Kingdom}
\email{e.zatorska@imperial.ac.uk}



\numberwithin{equation}{section}

\newtheorem{theorem}{Theorem}[section]

\newtheorem{corollary}{Corollary}[section]
\newtheorem{proposition}{Proposition}[section]
\newtheorem{remark}{Remark}[section]

\newcommand{\R}{\mathbb R}

\newcommand{\ls}{\lesssim}

\newcommand{\N}{\mathbb N}

\newcommand{\mc}{\mathcal C}

\newcommand{\bq}{\begin{equation}}
\newcommand{\eq}{\end{equation}}

\newcommand{\lt}{\left}
\newcommand{\rt}{\right}

\newcommand{\pa}{\partial}

\newcommand{\mms}{\mathcal{S}}
\newcommand{\om}{\Omega}
\newcommand{\coss}{\cos\lt( \frac{\sqrt\square}{2} t\rt)}
\newcommand{\cosss}{\cos\lt( \frac{\sqrt\square}{2} t^*\rt)}
\newcommand{\sins}{\sin\lt( \frac{\sqrt\square}{2} t\rt)}
\newcommand{\sinss}{\sin\lt( \frac{\sqrt\square}{2} t^*\rt)}
\newcommand{\tanss}{\tan\lt( \frac{\sqrt\square}{2} t^*\rt)}

\newcommand{\Deltaa}{\Xi}
\newcommand{\etan}{\eta^{n+1,n}}
\newcommand{\vn}{v^{n+1,n}}

\newcommand{\eqv}[1]{\begin{equation}
\begin{split}
#1
\end{split}
\end{equation}}
\newcommand{\eqh}[1]{\begin{equation*}
\begin{split}
#1
\end{split}
\end{equation*}}

\newcommand{\lr}[1]{\left(#1\right)}

\begin{document}
\allowdisplaybreaks

\date{\today}


\begin{abstract} We analyse the one-dimensional pressureless Euler-Poisson equations with a linear damping and non-local interaction forces. These equations are relevant for modelling  collective behavior in mathematical biology. We provide a sharp threshold between the supercritical region with finite-time breakdown and the subcritical region with global-in-time existence of the classical solution. We derive an explicit form of solution in Lagrangian coordinates which enables us to study the time-asymptotic behavior of classical solutions with the initial data in the subcritical region.  
\end{abstract}

\maketitle \centerline{\date}


%
%
%
%
\section{Introduction}
We are interested in the following 1D system of pressureless  Euler-Poisson equations with non-local interaction forces and damping:
\begin{align}\label{main_eq}
\begin{aligned}
&\pa_t \rho + \pa_x (\rho u) = 0, \cr
&\pa_t (\rho u) + \pa_x (\rho u^2) = -\rho u -(\pa W \star \rho)\rho, \qquad W(x) = - |x| + \frac{|x|^2}{2},
\end{aligned}
\end{align}
for $ (t,x) \in \R_+ \times \om(t)$. Here, $\rho$ is  extended by 0 outside $\om(t)$ and $\om(t)$ denotes the interior of the support of the density $\rho$, i.e., $\om(t) := \{x \in \R : \rho(x,t) >0\}$. System \eqref{main_eq} is supplemented by the initial values of the density and the velocity
\bq\label{ini_main_eq}
(\rho(t,\cdot), u(t,\cdot))|_{t=0} = (\rho_0,u_0)\in H^2(\Omega_0)\times H^3(\Omega_0),
\eq
where $H^s(\Omega_0)$ stands for the standard Sobolev space of index $s>0$ and $$\om_0:=\om(0)=(a_0,b_0)$$ is an open bounded interval. It follows from \eqref{ini_main_eq} that the initial mass and momentum are finite; we denote them by
$$
0<M_0 := \int_{\Omega_0}\rho_0(x) dx \qquad\mbox{and}
\qquad
M_1:=\int_{\om_0} \rho_0(x)u_0(x)\,dx  \,.
$$

The hydrodynamic system \eqref{main_eq} has been formally derived from interacting particle systems in collective dynamics. Different authors developed several approaches involving moment methods either for particle descriptions directly \cite{CDMBC} or at the kinetic level together with monokinetic closures for the pressure term \cite{CDP}. Kinetic equations for collective behavior can be derived rigorously from particle systems via the mean-field limit, see \cite{CCR,review} and the references therein. Although the monokinetic closure of the moment system is not entirely justified, these pressuless hydrodynamic models as \eqref{main_eq} give qualitative numerical results comparable to the particle simulations of interacting agents, see \cite{CKMT,AP,KT} and the references therein.

Critical threshold phenomena for the one-dimensional Euler or Euler-Poisson system are studied in \cite{ELT, TW}. In particular, the damped Euler-Poisson system with a positive background state is considered in \cite{ELT} and sharp critical thresholds are obtained. For certain restricted multi-dimensional Euler-Poisson systems, we refer to \cite{LT1, LT2}.  In \cite{TT}, the critical thresholds were analysed for the so-called {\it Euler-alignemt system} which has a non-local velocity alignment force $F[\rho,u] = \psi \star (\rho u) - u (\psi \star \rho)$ with $\psi \geq 0$ instead of the linear damping and interaction force in \eqref{main_eq}. Note that if $\psi \equiv 1$, then the alignment force $F[\rho,u]$ becomes the linear damping under the assumption that the initial momentum is zero, i.e., $M_1 = 0$. These results were further improved in \cite{CCTT} by closing the gap between lower and upper thresholds. Other interaction forces, such as attractive/repulsive Poisson forces or general-type forces, are also taken into account in the Euler-alignment system in \cite{CCTT}. However, the critical thresholds with interaction forces were not sharp. In this work, we solve the problem with linear damping and Newtonian attractive forces by observing that the system \eqref{main_eq} has a very nice Lagrangian formulation allowing for explicit computations of the classical solutions. 

Associated to the fluid velocity $u(t,x)$, we define the characteristic flow $\eta(t,x)$ as
\bq\label{eq_tra}
\frac{d \eta(t,x)}{dt} = u(t,\eta(t,x)) \quad \mbox{with} \quad \eta(0,x) = x \in \om_0\,.
\eq
We first define a classical solution for our system \eqref{main_eq} with the initial data \eqref{ini_main_eq}. We say that $(\rho(t,x),u(t,x))$ is a classical local-in-time solution to \eqref{main_eq} with the initial data \eqref{ini_main_eq}, if there exists  time $T>0$ such that $\rho$ and $u$ are $\mc^1$ and $\mc^2$ respectively in the set $\{(t,x)\in [0,T) \times \Omega(t)\}$, the characteristics $\eta(t,x)$ associated to $u$ defined by \eqref{eq_tra} are diffeomorphisms for all $t\in [0,T)$ with $\Omega(t)=\eta(t, \Omega_0)$, and $\rho$ and $u$ satisfy pointwisely the equations \eqref{main_eq} in $\{(t,x)\in [0,T) \times \Omega(t)\}$ with initial data \eqref{ini_main_eq}. Here the time derivative at $t=0$ has to be understood as a one-side derivative. It is not difficult to see that that this definition ensures the equivalence between the classical solution of the system \eqref{main_eq}  and the classical solutions to its Lagrangian formulation  \eqref{lag_eq}, given below. We will elaborate more about it in the next section.

We now explain our strategy to find classical solutions to the system \eqref{main_eq}. In Section 2, we assume that $(\rho(t,x),u(t,x))$ is a classical local-in-time solution to \eqref{main_eq} with initial data \eqref{ini_main_eq} in order to find some explicit expression for the solution on the whole time interval of existence $[0,T)$. Then, in Section 3, we analyse the maximal time interval of existence of the classical solution based on its explicit expression. 
We show that these solutions are in fact global-in-time classical solutions under certain hypotheses on the initial data, and that otherwise they  blow up in a finite time. In the end of Section 3, we state our main theorem, Theorem \ref{Th:main}, which gives sharp critical thresholds for the system \eqref{main_eq}.  Further, in Section 4, we describe the long time asymptotic behaviour of the classical global-in-time solutions. We show that the limit profile for the density is a sharp discontinuous function:
$$
\rho_\infty(x)= \frac{M_0}2 \quad \mbox{ and } \quad u_\infty(x)=0 \quad \mbox{for }\quad x\in \Omega_\infty:= ( \Gamma- 1,\Gamma + 1)
$$ 
with
\bq\label{newintro}
\Gamma:= \frac{1}{M_0}\lt(\int_\R x\rho_0(x)\,dx + \int_\R \rho_0(x) u_0(x)\,dx \rt)\,.
\eq
Let us point out that Theorem \ref{Th:main} also holds  in the whole space for positive integrable initial density with finite initial center of mass and finite initial mean momentum. However, we cannot ensure that their long time asymptotic behavior is given by $\rho_\infty$.
In Appendix A, for the sake of completeness, we provide a local-in-time existence and uniqueness result of classical solutions in the sense used in this paper.

Let us emphasize, that the explicit solutions constructed in our paper are proven to be the only classical solutions of the system \eqref{main_eq}. The local-in-time existence and uniqueness of classical solutions to the Euler-Poisson system is known for the initial data being a small perturbation of the stationary state, see \cite{M,MP}. There, the authors assume that the density is positive on the whole line $\R$ and that it tends to zero as $x\to \pm\infty$. We are not aware of any result for local-in-time well-posedness of the pressureless Euler-Poisson system neither for a Cauchy problem, nor for a bounded interval. Therefore, our local-in-time existence and uniqueness result for classical solutions to \eqref{main_eq} makes the construction of solutions from Sections 2 and 3 complete and justifiable. Strictly speaking, they are the only classical solutions in their maximal time interval of existence. Let us also observe that, in contrast to \cite{M,MP,ELT}, our results hold for the case of compactly supported initial data.

%
%
%
%

\section{Explicit expressions of classical solutions} \label{Sec:2}
Let us denote $f(t,x) := \rho(t,\eta(t,x))$ and $v(t,x) := u(t,\eta(t,x))$. Using the characteristic flow, it is easy to check that $(\rho,u)$ is a local-in-time classical solution of the system \eqref{main_eq} with initial data \eqref{ini_main_eq} if and only if $(f,v)$ is a classical solution of the system
  \begin{subequations}\label{lag_eq}
    \begin{align}\label{lag_eq1}
      f(t,x)\frac{\pa \eta(t,x)}{\pa x} &= \rho_0(x), \\
    v'(t,x) + v(t,x) &=-\int_{\om(t)} \pa W(\eta(t,x) - y) \rho(t,y)dy \nonumber\\
      &= -\int_{\om_0} \pa W(\eta(t,x) - \eta(t,y))\rho_0(y)\,dy, \label{lag_eq2}
    \end{align}
  \end{subequations}
for $(t,x) \in (0,\infty) \times \om_0$, 
where we used the conservation of mass $\eqref{lag_eq1}$ to fix the domain of integration in the right hand side of the equation $\eqref{lag_eq2}$. Here $\{ \}'$ denotes the time derivative along the characteristic flow $\eta$. The system \eqref{lag_eq} is supplemented with the initial data
\eqv{\label{ini_lag_eq}
f_0 := f(0,x) = \rho_0(x),\quad  v_0:=v(0,x) = u_0(x).}
Since $(\rho_0,u_0)\in H^2(\Omega_0)\times H^3(\Omega_0)$ and we are in one dimension, the initial data $\rho_0$ and $u_0$ are continuous functions up to the boundary of the domain, i.e., $\rho_0,u_0\in \mc([a_0,b_0])$.

The problem \eqref{lag_eq}-\eqref{ini_lag_eq} has a unique local-in-time classical solution according to Theorem \ref{thm_local} in Appendix A. This solution can be extended to a maximal time of existence of the classical solution $[0,T)$. Since the characteristic flow $\eta(t,x)$ is a diffeomorphism for all $t\in[0,T)$ such that $\Omega(t)=\eta(t, \Omega_0)$, the Lagrangian change of variables can be inverted and the corresponding $(\rho,u)$ are a local-in-time classical solution of \eqref{main_eq}-\eqref{ini_main_eq} in the sense given in the introduction. As mentioned above, we will now obtain explicitly the formulas for the classical solutions of the system \eqref{main_eq} in Lagrangian variables.

Observe that the equation for the density $f(t,x)$ is decoupled from the equation of the velocity variable $v(t,x)$. We  first with the equation for $v(t,x)$, and come back to the expression for the deformation of the mass density $\pa_x\eta(t,x)$ later on. Since the second derivative of the potential $\pa_x^2 W(x) = -2 \delta_0(x) + 1$ and $v\in\mc^2$, we find
$$\begin{aligned}
v''(t,x) + v'(t,x) &= -\int_{\om_0} \pa^2 W(\eta(t,x) - \eta(t,y))\lt( v(t,x) - v(t,y) \rt) \rho_0(y)\,dy \cr
&= - v M_0 + \int_{\om_0} v(t,y) \rho_0(y)\,dy.
\end{aligned}$$
To evaluate the second term on the right hand side of the above equation, we multiply $\eqref{lag_eq2}$ by $\rho_0$ and integrate with respect to $x$ to get
\[
\frac{d}{dt}\int_{\om_0} v(t,x) \rho_0(x)\,dx = - \int_{\om_0} v(t,x) \rho_0(x)\,dx,
\]
due to $\pa_x W(-x) = -\pa_x W(x)$, thus, using the initial condition \eqref{ini_main_eq} we conclude 
\eqv{\label{mom}
\int_{\om_0} v(t,x) \rho_0(x)\,dx = e^{-t}\int_{\om_0} \rho_0(x)u_0(x)\,dx.
}
Set $M_1 := \int_{\om_0} \rho_0(x)u_0(x)\,dx$. Then we obtain that $v$ satisfies the following nonhomogeneous linear second-order differential equation:
\bq\label{eq_diff}
v'' + v' + M_0 v = M_1 e^{-t}, \quad t > 0, \quad v_0 = u_0.
\eq
We notice that the initial data $v'(0,x)=v'_0(x)$ are given through the equation $\eqref{lag_eq2}$ by
\eqv{\label{vt0}
v_0'(x) &= - v_0(x) - \int_{\om_0} \pa W(x-y) \rho_0(y)\,dy\cr
&= -u_0(x) - \int_{\om_0} (x-y) \rho_0(y)\,dy + \int_{\om_0} sgn(x-y) \rho_0(y)\,dy\cr
&= -u_0(x) - (x+1) M_0 + \int_{\om_0} y \rho_0(y)\,dy + 2\int_{-\infty}^x \rho_0(y)\,dy  \quad \mbox{for} \quad x \in \om_0.
}
Depending on the size of the initial mass $M_0$, as long as the solution exists, it satisfies:\\
$\bullet$ {\bf Case A} ($1 > 4M_0$):
\eqv{\label{sol_formA}
v(t,x) =  C_1 e^{\lambda_1  t} + C_2 e^{\lambda_2 t} + \frac{M_1}{M_0} e^{-t} ,}
$\bullet$ {\bf Case B} ($1 =4M_0$):
\eqv{\label{sol_formB}
v(t,x) =C_3 e^{-t/2} + C_4 t \,e^{-t/2} + \frac{M_1}{M_0} e^{-t},}
$\bullet$ {\bf Case C} ($1 <4M_0$):
\eqv{\label{sol_formC}
v(t,x) =C_5 e^{-t/2} \cos\lt( \frac{\sqrt{4M_0 - 1}}{2}t\rt) + C_6 e^{-t/2} \sin\lt( \frac{\sqrt{4M_0 - 1}}{2}t\rt) + \frac{M_1}{M_0} e^{-t},}
where $\lambda_1$, $\lambda_2$, and $C_i,i=1,\cdots,6$ are given by
\begin{subequations}\label{coeff}
\begin{align}
\lambda_1 &:= \frac{-1 + \sqrt{1 - 4M_0}}{2}, \quad \lambda_2 := \frac{-1 - \sqrt{1 - 4M_0}}{2}, \label{coeff_l}\\
C_1 &:= \frac{1}{\lambda_2 - \lambda_1}\lt( \lambda_2 v_0 - v_0' + \lambda_1\frac{M_1}{M_0}\rt), \quad 
C_2 := \frac{1}{\lambda_2 - \lambda_1}\lt( -\lambda_1 v_0 + v_0' - \lambda_2\frac{M_1}{M_0}\rt), \label{coeff_12}\\
 C_3 &:= v_0 - \frac{M_1}{M_0}, \quad C_4:= \frac{v_0}{2} + v_0' + \frac{M_1}{2M_0},\label{coeff_34}\\
C_5&:= v_0 - \frac{M_1}{M_0}, \quad \mbox{and} \quad C_6 = \frac{2}{\sqrt{4M_0 - 1}}\lt(v_0' + \frac{v_0}{2} + \frac{M_1}{2M_0}\rt).\label{coeff_56}
\end{align}
\end{subequations}
For abbreviation, we set 
$$
\Deltaa := 1 - 4M_0\qquad \mbox{and} \qquad \square := -\Deltaa \,.
$$

Our aim now is to compute an explicit form of $\pa_x v$, in each of the above cases. Note that for any of these cases, it follows from \eqref{eq_tra} that
\begin{equation}\label{epe}
\eta(t,x) = x + \int_0^t v(s,x)\,ds \quad \mbox{and} \quad \pa_x \eta(t,x) = 1 + \int_0^t \pa_x v(s,x)\,ds.
\end{equation}
$\bullet$ {\bf Case A }($1 - 4M_0 > 0$): A straightforward computation for \eqref{sol_formA} yields
\bq\label{case1}
\pa_x v = \pa_x C_1 e^{\lambda_1  t} + \pa_x C_2 e^{\lambda_2 t},
\eq
and thus
\[
\pa_x v_0 = \pa_x C_1 + \pa_x C_2 \quad \mbox{and} \quad \pa_x v_0' = \pa_x C_1 \lambda_1 + \pa_x C_2 \lambda_2.
\]
On the other hand, it follows from \eqref{ini_lag_eq} and \eqref{vt0} that
\[
\pa_x v_0 = \pa_x u_0 \quad \mbox{and} \quad \pa_x v_0' = -\pa_x u_0 - M_0 + 2 \rho_0,
\]
which implies
\eqv{\label{case1_est1}
\pa_x C_1 = \frac{1}{\sqrt{\Deltaa}} \lt( \lambda_1 \pa_x u_0 - M_0 + 2\rho_0\rt) \quad \mbox{and} \quad \pa_x C_2 = \frac{1}{\sqrt{\Deltaa}} \lt( M_0 - 2\rho_0 - \lambda_2 \pa_x u_0\rt).
}
Combining \eqref{epe} with \eqref{case1}, we get
\bq\label{lt_etaA}
\pa_x \eta = 1 + \frac{\pa_x C_1}{\lambda_1}\lt( e^{\lambda_1 t} - 1 \rt) + \frac{\pa_x C_2}{\lambda_2}\lt(e^{\lambda_2 t} - 1 \rt) = \frac{2\rho_0}{M_0} + \frac{\pa_x C_1}{\lambda_1}e^{\lambda_1 t} + \frac{\pa_x C_2}{\lambda_2}e^{\lambda_2 t},
\eq
\bq\label{etaA}
\eta= x+\frac{C_1}{\lambda_1}(e^{\lambda_1 t}-1)+\frac{C_2}{\lambda_2}(e^{\lambda_2 t}-1)-\frac{M_1}{M_0}(e^{-t}-1),
\eq
with $C_1, C_2$ are given by \eqref{coeff_12} whose derivatives are computed in \eqref{case1_est1} and $\lambda_1,\lambda_2$ given by \eqref{coeff_l}.

\medskip

\noindent$\bullet$ {\bf Case B} ($1 = 4M_0$): We use again the solution to \eqref{eq_diff} given in \eqref{sol_formB} together with the initial conditions to get
\eqv{\label{case2}
 \quad \pa_x v = \pa_x C_3 e^{-t/2} + \pa_x C_4 t \,e^{-t/2},
}
where $\pa_x C_3,\ \pa_x C_4$ satisfy
\bq\label{case2_est1}
\pa_x C_3 = \pa_x u_0 \quad \mbox{and} \quad \pa_x C_4 = -\frac12 \pa_x u_0 - \frac14 + 2\rho_0\,,
\eq
and so, by \eqref{epe}, we find
\bq\label{lt_etaB}
\pa_x \eta = 8\rho_0 - \lt(2\pa_x C_3 + 4\pa_x C_4 \rt)e^{-t/2} - 2\pa_x C_4 t \,e^{-t/2}
\eq
and
\bq\label{etaB}
 \eta =x+2C_3(1-e^{t/2})-2 C_4 t e^{-t/2}+4C_4(1-e^{-t/2})+\frac{M_1}{M_0}(1-e^{-t}).
\eq

\medskip

\noindent$\bullet$ {\bf Case C} ($1 - 4M_0 < 0$): It follows analogously from \eqref{sol_formC} that
\eqv{\label{case3}
\pa_x v(t,x) = \pa_x C_5(x) e^{-t/2} \coss + \pa_xC_6(x) e^{-t/2} \sins,
}
where $\pa_x C_5,\ \pa_x C_6$ satisfy
\eqv{\label{case3_est1}
\pa_x C_5 = \pa_x u_0 \quad \mbox{and} \quad \pa_x C_6=\frac{2}{\sqrt\square} \lt(-\frac12 \pa_x u_0 - M_0 + 2\rho_0\rt).
}
This yields
\eqv{\label{lt_etaC}
\pa_x \eta &= \frac{2\rho_0}{M_0} + \lt( \frac{2\square}{1 + \square}\rt)\lt(\frac{\pa_x C_5}{\sqrt\square} - \frac{\pa_x C_6}{\square}\rt)e^{-t/2}\sins \cr
&\quad - \lt( \frac{2\square}{1 + \square}\rt)\lt(\frac{\pa_x C_5}{\square} + \frac{\pa_x C_6}{\sqrt\square}\rt)e^{-t/2}\coss,
}
and
\eqv{\label{etaC}
\eta &=x+ \frac{2(\sqrt\square C_5-C_6)}{1+\square} e^{-t/2}\sins+
\frac{C_5+\sqrt\square C_6}{1+\square}\lr{2-2e^{-t/2}\coss}\cr
&\quad+\frac{M_1}{M_0}(1-e^{-t}).
}

Let us summarize our results up to this point. We have  derived the explicit forms of  velocity field being a local-in-time classical solution to \eqref{main_eq}. We have also obtained the expressions for the deformation of the mass density $\pa_x \eta$ leading to positive values of the Lagrangian density $f(t,x)$ for small enough time, since $\pa_x \eta(0,x)=1$, for $x\in \Omega_0$. Moreover, we have derived the explicit expression of the characteristic flow $\eta(t,x)$. We next want to find the maximal time of existence of these explicit solutions.

%
%
%
%

\section{Sharp critical thresholds} In this section, we study the critical thresholds leading to a sharp condition for the dichotomy between global-in-time existence and finite-time blow-up of classical solutions to \eqref{main_eq}. The argument is based on the observation that the local-in-time classical solution found in the previous section can be extended in time as long as the characteristics can be defined, i.e., there is no crossing of characteristics, or equivalently, the flow map $\eta(t,x)$ is a diffeomorphism, so $\pa_x\eta>0$. We will thus study the explicit forms of $\pa_x\eta$  obtained in cases A, B and C above. The form of the time derivative of $\pa_x\eta$ will enable to estimate the critical thresholds in the system \eqref{lag_eq} depending on the size of the initial mass $M_0$. 

We first notice that for all cases A, B, and C, the global-in-time classical solution, if it exists,  satisfies
$$
\pa_x \eta(0,x) = 1 \quad \mbox{and} \quad \lim_{t \to \infty}\pa_x \eta(t,x) = \frac{2\rho_0(x)}{M_0}>0 \mbox{ for all } x\in \Omega_0.
$$
Thus, if the infimum of $\pa_x \eta(t,x)$ is nonpositive, then it should be attained at $0<t^*<\infty$. Let us assume that there exist $t^*>0$ and $x^*\in \Omega_0$ satisfying 
\bq\label{inf}
\pa_x \eta(t^*, x^*) = \inf_{t >0, \, x \in \om_0}\pa_x \eta(t,x)\leq 0.
\eq
Then using \eqref{epe} we find the necessary condition
\[
\partial_t \pa_x \eta(t^*,x^*)= \pa_x v(t^*,x^*) = 0.
\]

\noindent $\bullet$ {\bf Case A} ($1 - 4M_0 > 0$): Since $\lambda_1,\lambda_2$ given by \eqref{coeff_l} are both negative, it is clear from  \eqref{lt_etaA} that $\pa_x C_1(x^*) \pa_x C_2(x^*) \neq 0$ in order to have the infimum inside the time interval $(0,\infty)$. From \eqref{case1}  we also get
\[
\pa_x v(t^*,x^*) = \pa_x C_1(x^*) e^{\lambda_1 t} \lt( 1 + \frac{\pa_x C_2(x^*)}{\pa_x C_1(x^*)} e^{-\sqrt{\Deltaa} t^*}\rt) = 0,
\]
for
\eqv{\label{tsA}
-\frac{\pa_x C_1(x^*)}{\pa_x C_2(x^*)} = e^{-\sqrt{\Deltaa} t^*}.
}
This implies
\bq\label{case1_est2}
0 < -\frac{\pa_x C_1(x^*)}{\pa_x C_2(x^*)} < 1.
\eq
Further, from \eqref{lt_etaA} and \eqref{tsA} we obtain
\[
\pa_x \eta(t^*,x^*) = \frac{2\rho_0(x^*)}{M_0} + \frac{\pa_x C_1(x^*)}{\lambda_1}e^{\lambda_1 t^*} + \frac{\pa_x C_2(x^*)}{\lambda_2}e^{\lambda_2 t^*} = \frac{2\rho_0(x^*)}{M_0} + \frac{\sqrt{\Deltaa}}{M_0}\pa_x C_2(x^*) e^{\lambda_2 t^*},
\]
thus necessarily  $\pa_x C_2(x^*) < 0$ due to \eqref{inf}. Further, if $\pa_x C_2(x^*) < 0$, then due to \eqref{case1_est2},  \eqref{case1_est1} and \eqref{coeff_12} we have $\pa_x u_0(x^*) < 0$ which is equivalent to $\pa_x C_1(x^*) + \pa_x C_2(x^*) < 0$. Thus we conclude that to have finite-time blow up there must exist $x^* \in \om_0$ such that
\[
\pa_x C_1(x^*) > 0, \quad  \pa_x C_2(x^*) < 0, \quad \pa_x u_0(x^*) < 0,
\]
and
\bq\label{condii}
2\rho_0(x^*) + \sqrt{\Deltaa}\pa_x C_2(x^*)\lt(-\frac{\pa_x C_1(x^*)}{\pa_x C_2(x^*)}\rt)^{{\frac{\lambda_2}{-\sqrt{\Deltaa} }}} \leq 0.
\eq
The above condition is not only necessary but also sufficient, more precisely we have the following proposition:
\begin{proposition}\label{prop_c1} Suppose $1 - 4M_0 > 0$. Then $\pa_x \eta(t,x)$ attains a non-positive value if and only if there exists a $x \in \om_0$ such that 
\[
\pa_x u_0(x) < 0, \quad M_0 - 2\rho_0(x) < \lambda_1 \pa_x u_0(x),
\]
and
\[
2\rho_0(x) \leq (\lambda_1 \pa_x u_0(x) - M_0 + 2\rho_0(x))^{-\lambda_2/\sqrt{\Deltaa}}(\lambda_2 \pa_x u_0(x) - M_0 + 2\rho_0(x))^{\lambda_1/\sqrt{\Deltaa}}.
\]
\end{proposition} 
\begin{proof} Note that $M_0 - 2\rho_0(x) < \lambda_1 \pa_x u_0(x)$ is equivalent to $\pa_x C_1(x) > 0$ and $\pa_x C_2(x) < 0$ due to $\pa_x u_0(x) < 0$. Finally, it follows from \eqref{case1_est1} and \eqref{condii} that
\[
2\rho_0(x) \leq (\lambda_1 \pa_x u_0(x) - M_0 + 2\rho_0(x))^{-\lambda_2/\sqrt{\Deltaa}}(\lambda_2 \pa_x u_0(x) - M_0 + 2\rho_0(x))^{\lambda_1/\sqrt{\Deltaa}}.
\]
\end{proof}

\noindent $\bullet$ {\bf Case B} ($1 = 4M_0$): In this case, $\pa_x\eta$ is given by \eqref{lt_etaB} and \eqref{case2_est1}.
We again want to find a point $x^*$ which makes $\pa_x \eta$ nonpositive at some time $t=t^*$. Let us look for the values $t^*, x^*$ satisfying  $\pa_x v(t^*,x^*) = 0$, from \eqref{case2}, we have
\[
\pa_x v(t^*,x^*) = \pa_x C_3(x^*) e^{-t^*/2} + \pa_x C_4(x^*) t^* e^{-t^*/2}=0, \quad \mbox{i.e.,} \quad t^* = -\frac{\pa_x C_3(x^*)}{\pa_x C_4(x^*)}.
\]
Since we look for $t^*>0$ we must have $-\frac{\pa_x C_3(x^*)}{\pa_x C_4(x^*)} > 0$. On the other hand, by plugging $t^*$ and $x^*$ into \eqref{lt_etaB}, we get
\bq\label{neww}
\pa_x \eta(t^*,x^*) = 8\rho_0(x^*) - 4\pa_x C_4(x^*) e^{-t^*/2}.
\eq
Thus $\pa_x \eta(t^*,x^*)$ can be nonpositive if and only if 
\[
\pa_x C_4(x^*) > 0, \quad \pa_x C_3(x^*) < 0, \quad \mbox{and} \quad 2\ln \lt(\frac{2\rho_0(x^*)}{\pa_x C_4(x^*)} \rt) \leq \frac{\pa_x C_3(x^*)}{\pa_x C_4(x^*)}.
\]
Summarizing the above estimate together with \eqref{case2_est1}, we have the following proposition:
\begin{proposition} Suppose $1 = 4M_0$. Then $\pa_x \eta(t,x)$ attains a nonpositive value if and only if there exists a $x \in \om_0$ such that 
\[
\pa_x u_0(x) < \min\lt\{0, 4\rho_0(x) - \frac12\rt\},
\]
and
\bq\label{condii2}
\ln \lt(\frac{8\rho_0(x)}{8\rho_0(x) - 2\pa_x u_0(x) - 1}\rt) \leq \frac{2\pa_x u_0(x)}{8\rho_0(x) - 2\pa_x u_0(x) - 1}.
\eq
\end{proposition}
\begin{proof} Since $\pa_x C_3(x) =\pa_x u_0(x) < 0$ and $\pa_x C_4(x) > 0$, we infer
\[
\pa_x u_0(x) < \min\lt\{0, 4\rho_0(x) - \frac12\rt\}.
\]
The condition \eqref{condii2} just follows from \eqref{neww}, since $t^*>0$.
\end{proof}

\noindent $\bullet$ {\bf Case C} ($1 - 4M_0 < 0$): In this case, $\pa_x v$ is given by \eqref{case3}.
Let us look for the values $t^*,x^*$ satisfying $\pa_x v(t^*,x^*) = 0$, we have
\bq\label{est_c3_2}
\cosss = -\frac{\pa_x C_6(x^*)}{\pa_x C_5(x^*)} \sinss, \quad \mbox{i.e.,} \quad  -\frac{\pa_x C_5(x^*)}{\pa_x C_6(x^*)} = \tanss.
\eq
This gives 
\eqv{\label{est_c3_1}
&\pa_x \eta(t^*,x^*) \\
&\quad= \frac{2\rho_0(x^*)}{M_0} + \lt( \frac{2\sqrt\square}{1 + \square}\rt)\lt(\frac{(\pa_x C_5(x^*))^2 + (\pa_x C_6(x^*))^2}{\pa_x C_5(x^*)}\rt)e^{-t^*/2}\sinss,
}
due to \eqref{lt_etaC}. Note that the second term in the right hand side of the equality \eqref{est_c3_1} has a damped oscillatory behavior as a function of $t^*$. This implies that in order to get the minimum value of $\pa_x \eta(t^*,x^*)$, it is enough to find the point $x^* \in \om_0$ and the smallest time $t^* > 0$ satisfying \eqref{est_c3_2}, such that the sign of the second term in  \eqref{est_c3_1} is negative, i.e. $\sin(\sqrt{\square}\, t^*/2) \pa_x C_5(x^*) < 0$. Observe that for each $x^*\in \Omega_0$, there is an increasing sequence of allowed positive $t^*$ due to condition \eqref{est_c3_2}. For this, we consider the following two cases:\\

\noindent {\bf Subcase C.1} $\pa_x C_5(x^*)\pa_x C_6(x^*)< 0$: It follows from \eqref{est_c3_2} that the first $t^* > 0$ satisfying \eqref{est_c3_2} appears in the interval $(0, \pi/\sqrt\square)$. This yields that  $\sin(\sqrt{\square}\, t^*/2) > 0$, therefore we can further distinguish two different cases:\\

\noindent {\bf Subcase C.1.i} If in addition $\pa_x C_5 <0$, it is possible that the first $t^* > 0$ satisfying \eqref{est_c3_2} leads to a negative value of \eqref{est_c3_1}. We can write its form in an explicit way; due to \eqref{est_c3_2} we have
\[
\sin\lt(\frac{\sqrt\square}{2} t^*\rt) = - \frac{\pa_x C_5(x^*)}{\sqrt{(\pa_x C_5(x^*))^2 + (\pa_x C_6(x^*))^2}} > 0.
\]
Plugging this into \eqref{est_c3_1}, we get
\[
\pa_x \eta(t^*,x^*) = \frac{2\rho_0(x^*)}{M_0} - \lt( \frac{2\sqrt\square}{1 + \square}\rt)\sqrt{(\pa_x C_5(x^*))^2 + (\pa_x C_6(x^*))^2}\,e^{-t^*/2}.
\]
Then we again use the relation \eqref{est_c3_2} to find
\[
\pa_x \eta(t^*,x^*) = \frac{2\rho_0(x^*)}{M_0} - C_7(x^*)\exp\lt( \frac{C_8(x^*)}{\sqrt\square}\rt),
\]
where $C_7$ and $C_8$ are given by
\bq\label{case3_c2}
C_7(x) := \lt( \frac{2\sqrt\square}{1 + \square}\rt)\sqrt{(\pa_x C_5(x))^2 + (\pa_x C_6(x))^2}  \mbox{ and }  C_8(x) := \arctan\lt(\frac{\pa_x C_5(x)}{\pa_x C_6(x)}\rt).
\eq

\noindent {\bf Subcase C.1.ii} If in addition $\pa_x C_5 >0$, then the first $t^* > 0$ satisfying \eqref{est_c3_2} leads to a positive value of \eqref{est_c3_1}, but the next $t^*$ might lead to a negative value. This one occurs at 
$$
t_1^*=t^*+\frac{2\pi}{\sqrt{\square}}\in ( 2\pi/\sqrt\square, 3\pi/\sqrt\square),
$$ 
for which $\sin(\sqrt{\square}\, t_1^*/2) <0$, however its form is still the same
\[
\sin\lt(\frac{\sqrt\square}{2} t_1^*\rt) = - \frac{\pa_x C_5(x^*)}{\sqrt{(\pa_x C_5(x^*))^2 + (\pa_x C_6(x^*))^2}} < 0,
\]
and thus
\[
\pa_x \eta(t_1^*,x^*) = \frac{2\rho_0(x^*)}{M_0} - C_7(x^*) \exp\lt( \frac{1}{\sqrt\square}\lt(C_8(x^*) - \pi\rt)\rt).
\]
{\bf Subcase C.2} $\pa_x C_5(x^*)\pa_x C_6(x^*) > 0$: In this case, the first $t^* > 0$ satisfying \eqref{est_c3_2} is later, namely $t^* \in ( \pi/\sqrt\square, 2\pi/\sqrt\square)$, however this gives again the positive value of $\sin(\sqrt{\square}\, t^*/2) >0$. Therefore, we can further distinguish similar two cases as in {\bf C.1}:\\

\noindent {\bf Subcase C.2.i} If in addition $\pa_x C_5 <0$, then the minimum value 
$\pa_x \eta(t^*,x^*)$ can be written in the following way
\[
\pa_x \eta(t^*,x^*) = \frac{2\rho_0(x^*)}{M_0} - C_7(x^*) \exp\lt( \frac{1}{\sqrt\square}\lt(C_8(x^*) - \pi\rt)\rt).
\]
Note that since $C_8(x^*) > 0$ and $\frac{\sqrt{\square}}{2}t^*> 0$, one has to take $\frac{\sqrt{\square}}{2}t^*= -C_8(x^*)+\pi$.

\noindent {\bf Subcase C.2.ii} If in addition $\pa_x C_5 >0$, then the minimum value is attained in the next possible time according to \eqref{est_c3_2} given by
$$
t_1^*=t^*+\frac{2\pi}{\sqrt{\square}}\in ( 3\pi/\sqrt\square, 4\pi/\sqrt\square),
$$ 
so, the smallest value is given by 
\[
\pa_x \eta(t_1^*,x^*) = \frac{2\rho_0(x^*)}{M_0} - C_7(x^*) \exp\lt( \frac{1}{\sqrt\square}\lt(C_8(x^*) - 2\pi\rt)\rt).
\]
All of these sub-cases for $1 - 4M_0 < 0$ can be summarized in the following result. 
\begin{proposition}\label{prop_c3} Suppose $1 - 4M_0 < 0$. Then $\pa_x \eta(t,x)$ has a nonpositive value if and only if there exists a point $x \in \mms_1 \cup \mms_2\cup\mms_3\cup\mms_4$ where $\mms_i,i=1,...,4$ are given by
\begin{align}\label{condi_c3}
\begin{aligned}
\mms_1 \!\!&:=\!\! \lt\{ x \in \om_0\!:\!\pa_x C_5(x) < 0, \pa_x C_6(x) >0, \frac{2\rho_0(x)}{M_0} -  C_7(x)\exp\!\lt( \frac{C_8(x)}{\sqrt\square}\rt) \!\!\leq 0 \rt\},\cr
\mms_2 \!\!&:=\!\! \lt\{ x \in \om_0\!:\! \pa_x C_5(x) > 0, \pa_x C_6(x) <0, \frac{2\rho_0(x)}{M_0} - C_7(x) \exp\!\lt( \frac{C_8(x) - \pi}{\sqrt\square}\rt) \!\!\leq 0 \rt\},\cr
\mms_3 \!\!&:= \!\!\lt\{ x \in \om_0\!:\! \pa_x C_5(x) < 0, \pa_x C_6(x) <0, \frac{2\rho_0(x)}{M_0} - C_7(x) \exp\!\lt( \frac{C_8(x) - \pi}{\sqrt\square}\rt) \!\!\leq 0 \rt\},\cr
\mms_4 \!\!&:=\!\! \lt\{ x \in \om_0\!:\! \pa_x C_5(x) > 0, \pa_x C_6(x) >0, \frac{2\rho_0(x)}{M_0} - C_7(x) \exp\!\lt( \frac{C_8(x) - 2\pi}{\sqrt\square}\rt) \!\!\leq 0 \rt\}
\end{aligned}
\end{align}
respectively. Here $\pa_x C_i,i=5,6,7,8$ are given in \eqref{case3_est1} and \eqref{case3_c2}.
\end{proposition}
As a direct consequence of Propositions \ref{prop_c1} -- \ref{prop_c3}, we have the following sharp critical thresholds for the system \eqref{main_eq}.

\begin{theorem}\label{Th:main} 
Assume that $(f,v)$ is a classical solution to the system \eqref{lag_eq} with initial data \eqref{ini_lag_eq}, then:

\noindent {\bf Case A:} If $1 - 4M_0 > 0$, the solution blows up in finite time if and only if there exists a $x^* \in \om_0$ such that
\[
\pa_x u_0(x) < 0, \quad M_0 - 2\rho_0(x) < \lambda_1 \pa_x u_0(x),
\]
and
\[
2\rho_0(x) \leq (\lambda_1 \pa_x u_0(x) - M_0 + 2\rho_0(x))^{-\lambda_2/\sqrt{\Deltaa}}(\lambda_2 \pa_x u_0(x) - M_0 + 2\rho_0(x))^{\lambda_1/\sqrt{\Deltaa}}.
\]
{\bf Case B:} If $1 - 4M_0 = 0$, the solution blows up in finite time if and only if there exists a $x^* \in \om_0$ such that
\[
\pa_x u_0(x) < \min\lt\{0, 4\rho_0(x) - \frac12\rt\},
\]
and
\[
\ln \lt(\frac{8\rho_0(x^*)}{8\rho_0(x^*) - 2\pa_x u_0(x^*) - 1}\rt) \leq \frac{2\pa_x u_0(x^*)}{8\rho_0(x^*) - 2\pa_x u_0(x^*) - 1}.
\]
{\bf Case C:} If $1 - 4M_0 < 0$, the solution blows up in finite time if and only if there exists a $x^* \in \mms_1 \cup \mms_2\cup\mms_3\cup\mms_4$ where $\mms_i,i=1,...,4$ are given in \eqref{condi_c3} and with $C_i(x), i=5,\cdots,8$ given by
\[
\pa_x C_5(x) = \pa_x u_0(x), \quad \pa_x C_6(x) = \frac{2}{\sqrt\square} \lt(-\frac12 \pa_x u_0(x) - M_0(x) + 2\rho_0(x)\rt),
\]
\[
C_7(x) = \lt( \frac{2\sqrt\square}{1 + \square}\rt)\sqrt{\frac{1 + \square}{\square}(\pa_x u_0(x))^2 + \frac{4}{\square}\lt(2\rho_0(x) - M_0\rt)\lt(2\rho_0(x) - M_0 - \pa_x u_0(x)\rt)},
\]
and
\[
C_8(x) = \arctan \lt( \frac{\sqrt{\square} \pa_x u_0(x)}{4\rho_0(x) - 2M_0 - \pa_x u_0(x)}\rt).
\]
Moreover, for all cases, if there is no finite-time blow-up, then the classical solution $(f,v)$ exists globally in time.
\end{theorem}
\begin{proof}It follows from $\eqref{lag_eq1}$ that 
\[
\rho(t,\eta(t,x)) = \rho_0(x) \lt(\pa_x \eta(t,x)\rt)^{-1}.
\]
Thus the density $\rho$ blows up if and only if $\inf_{x \in \om_0} \pa_x \eta(t,x) \leq 0$ for some finite time $t>0$. We finally use Propositions \ref{prop_c1} -- \ref{prop_c3} to conclude the desired result.
\end{proof}

\begin{remark}\label{rempos}
One can easily check that the previous theorem holds also for the case  $\Omega_0=\R$, provided that the initial density is positive and integrable and that the following conditions are satisfied
\eqv{\label{Rem_31}
\int_\R |x| \rho_0(x)\,dx <\infty \qquad \mbox{and} \qquad \int_\R \rho_0(x) |u_0(x)|\,dx <\infty \,.
}
\end{remark}

%
%
%
%

\section{Asymptotic behaviour}
The purpose of this section is to investigate the large time asymptotic behaviour of the explicitly constructed classical solutions to system \eqref{lag_eq} ensured by Theorem \ref{Th:main}.

\begin{theorem}\label{thm_large} Let $(f,v)$ be a global-in-time classical solution to the system \eqref{lag_eq}-\eqref{ini_lag_eq} given by Theorem  {\rm\ref{Th:main}}. Then it satisfies 
\[
f_\infty(x):=\lim_{t\to\infty} f(t,x) = \frac{M_0}{2} \quad \mbox{and} \quad v_\infty(x):=\lim_{t\to\infty} v(t,x) = 0 \quad \mbox{for all } x\in\Omega_0,
\]
exponentially fast. Moreover, the characteristic flow satisfies
$$
\eta_\infty(x) := \lim_{t\to\infty} \eta(t,x) = \frac{1}{M_0}\lt(\int_{\Omega_0}\!\! y\rho_0(y)\,dy +\!\! \int_{\Omega_0}\!\! \rho_0(y) \,u_0(y)\,dy + 2\int_{a_0}^x \!\!\rho_0(y)\,dy -M_0\rt)
$$
for all $x\in\Omega_0$. In particular, $\Omega(t)=(a(t),b(t))$ and
$$
\lim_{t\to\infty} |a(t)-\Gamma+1|=0 \quad \mbox{ and } \quad \lim_{t\to\infty} |b(t) -\Gamma-1|=0\,,
$$
exponentially fast.
\end{theorem}
\begin{proof} We claim that if there is no blow-up
\[
\lim_{t \to \infty}\frac{\pa \eta(t,x)}{\pa x} \to \frac{2\rho_0(x)}{M_0} \quad \mbox{for all} \quad x \in \om_0.
\]
It simply follows from the explicit formulas for $\pa_x\eta$ obtained in Section \ref{Sec:2}, namely \eqref{lt_etaA}, \eqref{lt_etaB}, and \eqref{lt_etaC}. On account of $\eqref{lag_eq1}$, we therefore have
\[
\lim_{t \to \infty} f(t,x) = \frac{M_0}{2} \quad \mbox{for} \quad x \in \om_0.
\]
Finally, it is obvious due to \eqref{coeff} that all functions $\pa_x C_i$ are bounded due to $\rho_0, \pa_x u_0\in \mc([a_0,b_0])$, and thus, there exists a constant $ C>0$ such that
\[
\max_{1 \leq i \leq 6}\|C_i(x)\|_{L^\infty(\Omega_0)} \leq C.
\]
This yields
$$
\|v(t,\cdot)\|_{L^\infty(\Omega_0)} \leq  C e^{-\lambda t} \quad \mbox{for some} \quad \lambda > 0.
$$

Since there is no blow-up of solution, we know that $\pa_x\eta(t,x)>0$ for all $(t,x)\in [0,\infty)\times\om_0$. Thus, $\rho(t,\eta(t,x))>0$ for all $(t,x)\in [0,\infty)\times\om_0$, and so, $\Omega(t)$ is connected since $\eta(t,x)$ is a diffeomorphism from the connected set $\om_0$ onto $\Omega(t)$. We denote $\Omega(t)=(a(t),b(t))$, where
$$
a(t)=\lim_{x\to a_0 +} \eta(t,x) \qquad \mbox{and} \qquad b(t)=\lim_{x\to b_0 -} \eta(t,x)\,.
$$

Finally, we can compute based on the explicit formulas for $\eta(t,x)$ given in \eqref{etaA}, \eqref{etaB}, and \eqref{etaC} that
$$
\lim_{t\to\infty} \eta(t,x) = \frac{1}{M_0}\lt(\int_{\Omega_0} y\rho_0(y)\,dy + \int_{\Omega_0} \rho_0(y) \,u_0(y)\,dy + 2\int_{a_0}^x \rho_0(y)\,dy -M_0\rt),
$$
for all $x\in\Omega_0$ and we deduce again that there exists constants $\bar C>0$ and $\bar \lambda>0$ such that 
\eqv{\label{ab1}
\lim_{t\to\infty} |a(t) -\Gamma+ 1|\leq \bar Ce^{-\bar\lambda t} \qquad \mbox{and} \qquad\lim_{t\to\infty} |b(t) -\Gamma- 1|\leq \bar Ce^{-\bar\lambda t}
}
with $\Gamma$ given in \eqref{newintro}.
\end{proof}

\begin{remark}
As a consequence of Theorem {\rm\ref{thm_large}}, we conclude
$$
\lim_{t\to\infty} \rho(t,\eta(t,x)) = \frac{M_0}{2} \quad \mbox{ and } \quad
\lim_{t\to\infty} u(t,\eta(t,x)) = 0
$$
for all $x\in\Omega_0$. We can also check that $\eta_\infty(x)$ is  a diffeomorphism from $\Omega_0$ to $(\Gamma-1,\Gamma+1)$. The previous theorem and this remark also hold for positive initial density defined on the whole $\R$ under the assumptions {\rm\eqref{Rem_31}}.
\end{remark}

In order to understand the large time behaviour of $\rho(t,y)$ in the Eulerian variables, one should invert the characteristics $\eta(t,x)$. This would be a daunting task in view of complexity of the explicit formulas for $\eta$ given in\eqref{etaA}, \eqref{etaB}, and \eqref{etaC} and we do not intend to do it. However, one can estimate the error in $L^1$ norm between $\rho(t,y)$ and the expected asymptotic profile
$$
\rho_\infty (y) = \frac{M_0}2 \chi_{\Omega_\infty}(y) \mbox{ for } y\in\R\,,
$$
where  $\chi_\Omega$ is the characteristic function of the interval $\Omega$, recall that $\Omega_\infty=(\Gamma-1,\Gamma+1)$. In order to estimate this difference, we define an intermediate function $\tilde\rho$ that will simplify our computations:
$$
\tilde\rho (t,y)=\frac{M_0}{2}\chi_{\Omega(t)}(y) \mbox{ for } y\in\R\,.
$$
By using the Lagrangian change of variables \eqref{lag_eq1}, we deduce
\eqv{\label{L1b}
\|\rho(t,\cdot)-\tilde\rho (t,\cdot)\|_{L^1(\R)} &= 
\int_{\Omega_0} \left| f(t,x)-\frac{M_0}{2} \right| \pa_x \eta(t,x) \,dx \cr &= 
\int_{\Omega_0} \left| \rho_0(x)-\frac{M_0}{2} \pa_x \eta(t,x)\right|  \,dx \,.
}
Theorem \ref{thm_large} shows that
$$
\lim_{t\to\infty} \left[\rho_0(x)-\frac{M_0}{2} \pa_x \eta(t,x)\right] = 0\,, \quad \mbox{for all} \quad x \in \om_0
$$
due to $\pa_x \eta > 0$. Since $\rho_0,\partial_x u_0 \in H^2(\Omega_0)$, then by the Sobolev embeddings $\rho_0,\partial_x u_0 \in \mc^1(\Omega_0)$, and thus by the  explicit expressions of $\pa_x\eta$ in \eqref{lt_etaA}, \eqref{lt_etaB} and \eqref{lt_etaC} we easily get
$$
\|\pa_x\eta\|_{L^\infty((0,\infty)\times \Omega_0)}<\infty\,.
$$
Therefore the integrand in \eqref{L1b} is bounded by a constant and the dominated convergence theorem implies that 
$$
\lim_{t\to\infty} \|\rho(t,\cdot)-\tilde\rho (t,\cdot)\|_{L^1(\R)}= 0\,.
$$
It is also true on account of \eqref{ab1} that
$$
\lim_{t\to\infty} \|\rho_\infty(\cdot)-\tilde\rho (t,\cdot)\|_{L^1(\R)}= 0\,.
$$
Putting together the above results, we have 
$$
\lim_{t\to\infty} \|\rho (t,\cdot)-\rho_\infty(\cdot)\|_{L^1(\R)}= 0\,.
$$
We can even improve this result providing a rate of convergence. 

\begin{corollary}\label{cor_large} Let $(\rho,u)$ be the global-in-time classical solution to the system \eqref{main_eq}-\eqref{ini_main_eq} given by Theorem  {\rm\ref{Th:main}}. Then there exists $C>0$ depending on the $L^\infty$ bounds of $\rho_0$ and $\pa_x u_0$ in $\Omega_0$ and $\lambda>0$ depending on the initial mass $M_0$ such that
$$
\|\rho (t,\cdot)-\rho_\infty(\cdot)\|_{L^1(\R)}\leq C e^{-\lambda t}\,.
$$
\end{corollary}

\begin{proof}
Using the explicit expressions for $\pa_x\eta$ in \eqref{lt_etaA}, \eqref{lt_etaB}, and \eqref{lt_etaC}, we can write the integrand in \eqref{L1b} as
$$
\rho_0(x)-\frac{M_0}{2} \pa_x \eta(t,x)=-\frac{M_0}{2} \xi(t,x)\,
$$
where 
$$
\xi(t,x):=
\begin{cases}
\frac{\pa_x C_1}{\lambda_1}e^{\lambda_1 t} + \frac{\pa_x C_2}{\lambda_2}e^{\lambda_2 t} & \mbox{in {\bf A}}\\[2mm]
- \lt(2\pa_x C_3 + 4\pa_x C_4 \rt)e^{-t/2} - 2\pa_x C_4 t \,e^{-t/2} & \mbox{in {\bf B}}\\[2mm]
\lt( \frac{2\square}{1 + \square}\rt)e^{-t/2}\lt[\lt(\frac{\pa_x C_5}{\sqrt\square} - \frac{\pa_x C_6}{\square}\rt)\sins 
- \lt(\frac{\pa_x C_5}{\square} + \frac{\pa_x C_6}{\sqrt\square}\rt)\coss\rt] & \mbox{in {\bf C}}
\end{cases}\,.
$$
Therefore, it is easy to check due to \eqref{coeff} that all functions $\pa_x C_i$ are bounded due to $\rho_0, \pa_x u_0\in \mc([a_0,b_0])$, and thus, there exists a constant $\tilde C>0$ such that
$$
\|\xi(t,\cdot)\|_{L^\infty(\Omega_0)}\leq \tilde C e^{-\tilde \lambda t}
$$
with
$$
\tilde \lambda:=
\begin{cases}
-\lambda_1 & \mbox{in {\bf A}}\\[2mm]
\tfrac12 -\epsilon & \mbox{in {\bf B}}\\[2mm]
\tfrac12 & \mbox{in {\bf C}}
\end{cases}\,,
$$
with $\epsilon>0$ arbitrarily small.
Using these estimates back in \eqref{L1b}, we get
\eqh{
\|\rho(t,\cdot)-\tilde\rho (t,\cdot)\|_{L^1(\R)} &\leq  \frac{|\Omega_0|}{2}\|\xi(t,\cdot)\|_{L^\infty(\Omega_0)}M_0
 \leq \frac{\tilde C M_0|\Omega_0|}{2} e^{- \tilde\lambda t}.
}
The remaining term is also straightforward to estimate, using \eqref{ab1} we have
\eqh{
\|\tilde\rho (t,\cdot)-\rho_\infty(\cdot)\|_{L^1(\R)}\leq \frac{M_0}2 \int_{\Omega_0} |\chi_{\Omega(t)}-\chi_{\Omega_\infty}| dx\leq \bar C e^{-\bar\lambda t}\,
}
and we conclude by taking
$$
C=\min\left\{\frac{\tilde C M_0|\Omega_0|}{2},\bar C\right\},\quad\mbox{and}\quad \lambda=\min\{\tilde\lambda,\bar \lambda\}.
$$
\end{proof}

\begin{remark}
Let us point out that one can give more qualitative estimate on the intermediate asymptotics of the solutions. Actually, one can prove as in Corollary {\rm\ref{cor_large}} that the $L^1$ difference between any solution and the density profile
$$
\bar\rho (t,y)=\frac{M_0}{|\Omega(t)|}\chi_{\Omega(t)}(y) \mbox{ for } y\in\R\,,
$$
converges exponentially fast to zero. Depending on the different time scales involved, one can have cases in which this tendency to adjust to $\bar\rho$ is faster initially before the solution finally relaxes to the global equilibrium $\rho_\infty$. Adapting the previous arguments for positive initial data under the assumptions in Remark {\rm\ref{rempos}} seems challenging. This needs a smart control of the tails of the solutions as $t\to\infty$ depending on decaying/growth conditions at $x=\pm\infty$ of the density and the velocity profiles.
\end{remark}

Let us illustrate the results of the last sections with some numerical experiments performed using a particle method to solve the Lagrangian equations \eqref{lag_eq}. We refer to \cite{CCP} for details on the numerical scheme, see also \cite{KT} for related numerical strategies. We use an initial uniform distribution of nodes given by
\[
\eta_i(0)=-0.75+\frac{1.5}{n-1}\left(i-1\right)\quad \mbox{for} \quad i=1,\cdots,n.
\]
\begin{figure}[ht!]
\quad\,\, \subfloat[]{
\protect\includegraphics[scale=0.32]{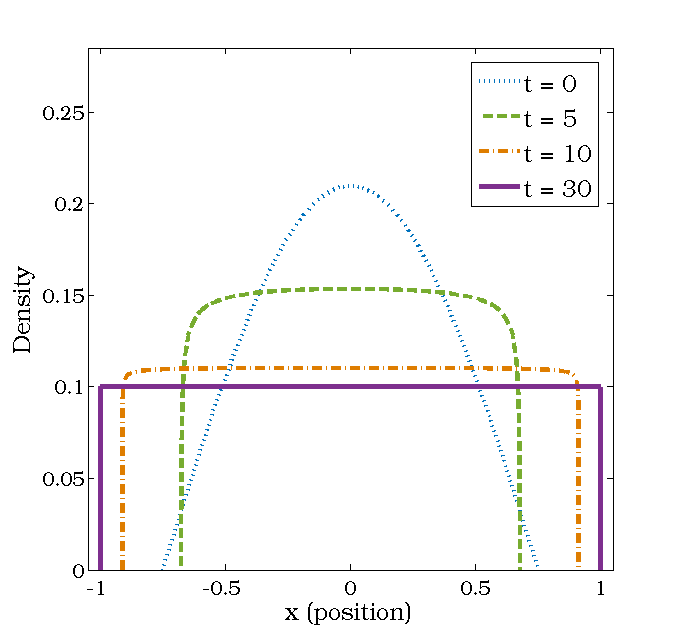}
\llap{\shortstack{%
        \includegraphics[scale=.09]{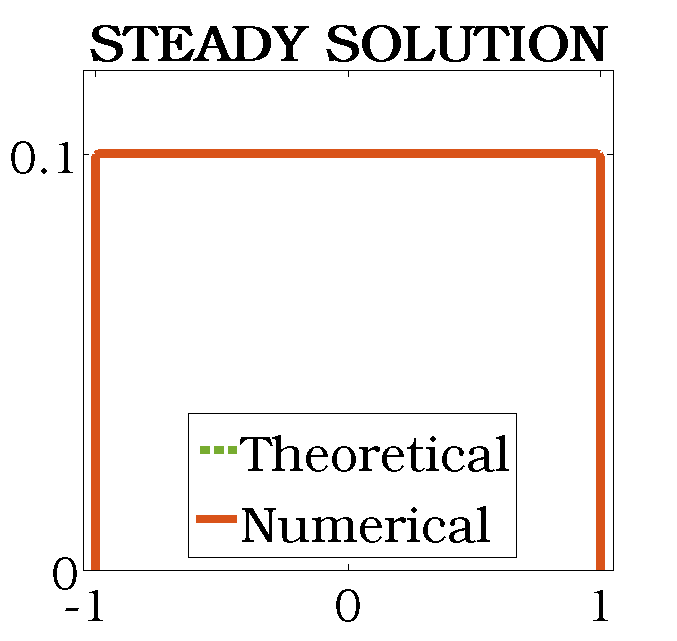}\\
        \rule{0ex}{1.3in}%
      }
  \rule{1.31in}{0ex}}
}\subfloat[]{\protect\includegraphics[scale=0.32]{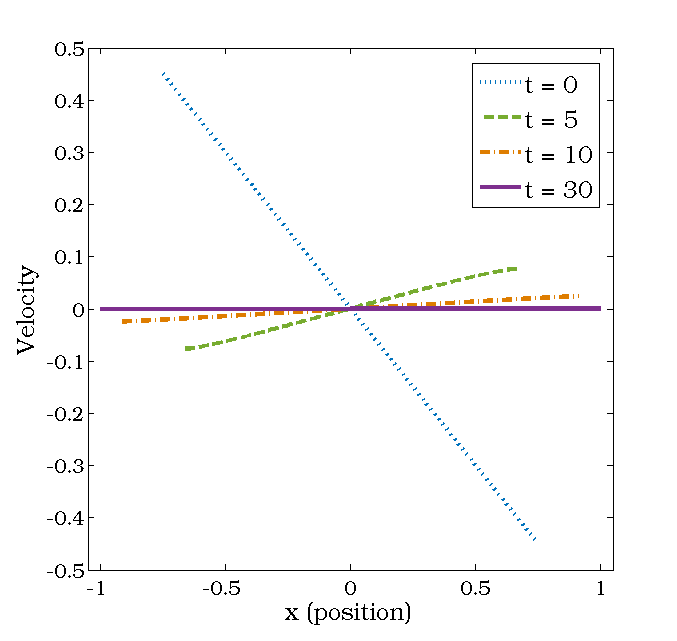}
}\newline
\subfloat[]{
\protect\includegraphics[scale=0.32]{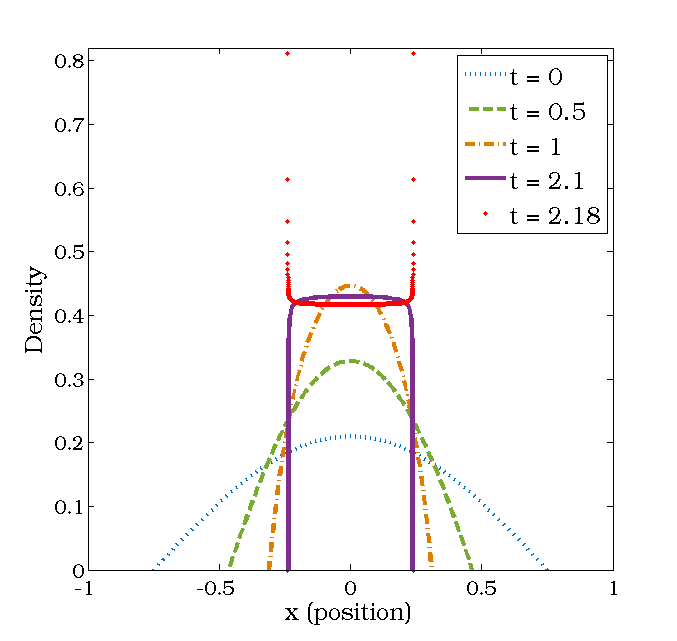}
}\subfloat[]{\protect\includegraphics[scale=0.32]{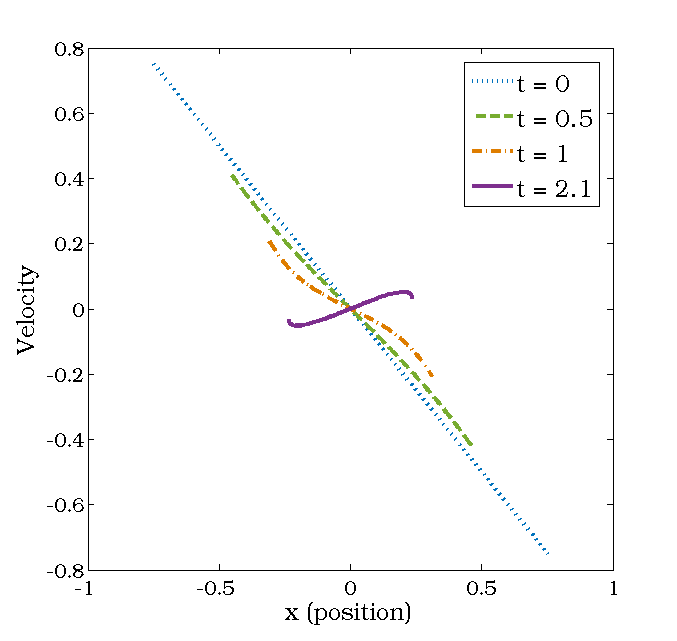}
\llap{\shortstack{%
        \includegraphics[scale=.1]{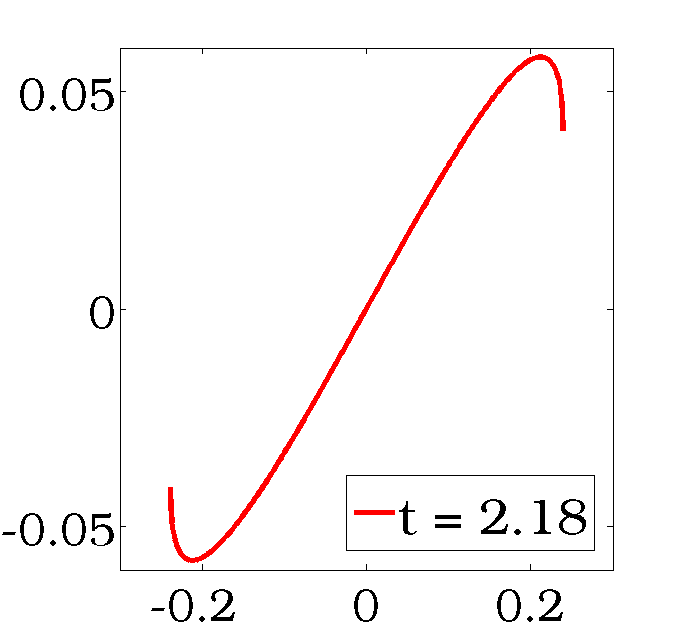}\\
        \rule{0ex}{0.23in}%
      }
  \rule{1.23in}{0ex}}
}\protect\caption{\label{fig:figure} Numerical simulation of the system \eqref{main_eq} in the Lagrangian variables.- (A), (B): Time behavior of the density and the velocity for a global existence case ($c=0.6$). (C), (D): Time behavior of the density and the velocity for a finite time blow-up case ($c=1$).}
\end{figure}
The initial density is chosen as 
\[
\rho_{i}(0)=\frac{1}{\gamma}\cos\left(\pi\,\frac{x_{i}(0)}{1.5}\right),
\]
where the constant $\gamma$ is fixed so that the total mass $M_0 := \int_\R \rho_0\,dx=0.2$. Concerning the initial velocity, we choose
\[
u_{i}(0)=-c\text{\,}x_{i}(0)\quad\mbox{for each node } i=1,\cdots, n,
\]
where the two values of the parameter $c$ will be $0.6$ and $1$. For the case $c=0.6$ there is global classical solution and for the case $c=1$ there is finite-time blow-up according to Theorem \ref{Th:main}. 

In Fig. \ref{fig:figure} (A) and (B), we observe the dynamics of the solution converging towards the asymptotic profile $\rho_\infty$ as $t$ gets larger while the velocity becomes zero everywhere in the support of $\rho$. The solution after $t=30$ is plotted against the asymptotic profile steady state $\rho_\infty$ in the inlet for further validation.

In Fig. \ref{fig:figure} (C) and (D), we show the dynamics of the solution in the blow-up case. In the density evolution, we observe how the density is squeezing towards the asymptotic profile up to certain time $t=2.1$, after which the density becomes larger and larger at the boundary. The blow-up is clearer in the velocity profile where we see that the derivative of the velocity becomes unbounded at the boundary at approximately $t=2.179$ as depicted in the inlet. At this time before several nodes have been removed for the density symmetrically near the boundary for visualization purposes, whose largest value is 27.94.

\begin{remark} Observe that the same asymptotic profile $\rho_\infty$ is obtained as the large time asymptotics of the first-order aggregation equation:
$$\begin{aligned}
& \pa_t \rho + \nabla_x \cdot (\rho u) = 0,\quad (t,x) \in \R_+ \times \R^d,\cr
& u = -\nabla_x W \star \rho,\quad W(x) = -\phi(x) + \frac{|x|^2}{2} \quad \mbox{where} \quad \Delta_x \phi = 2\delta_0.
\end{aligned}$$
Indeed, one can easily find the dynamics of $\rho$ along its characteristic flow. More precisely, we get
\[
f' = -f (\nabla_x \cdot u)(t,\eta(t,x)) = -f^2(2 - dM_0 f^{-1}).
\]
This and together with the Gronwall inequality yields
\[
f(t,x) = \frac{dM_0\rho_0 }{(dM_0 - 2\rho_0)e^{-dM_0t} + 2\rho_0} \to \frac{dM_0} {2} \quad \mbox{as} \quad t \to \infty,
\]
for some $x \in \om_0$. These facts were already analysed both theoretically and numerically in \cite{BLL} for the attractive and repulsive Newtonian potentials in any dimension. In fact, the aggregation equation can be formally understood as the large friction limit of \eqref{main_eq}, see \cite{LT} for related asymptotic limits. Let us also point out that this aggregation equation for Newtonian repulsive interaction can be obtained from particle dynamics \cite{BCDP}.
\end{remark}

\begin{remark}
Further extensions for potentials may be possible following the previous strategy. Let us consider a more repulsive force at the origin in our main system \eqref{main_eq} by defining the potential $W(x)$ to be
$$
W(x)=-\frac{|x|^\alpha}{\alpha}+\frac{x^2}{2},
$$
with $-1<\alpha<1$. Here, $\frac{|x|^0}{0}:=\log |x|$ by definition. It is well known that $\frac{|x|^\alpha}{\alpha}$ is the fundamental solution of the fractional operator $-(-\partial_{xx})^{(1+\alpha)/2}$ except a positive constant. More precisely, one can check that
$$
-(-\partial_{xx})^{(1+\alpha)/2} \lt(\frac{|x|^\alpha}{\alpha}\rt) = k\delta_0 
$$
with $k>0$, see \cite{Landkof,Stein} and \cite{CFP,ledoux,CHSV} for the one dimensional case. These potentials have been used for first-order aggregations models as in previous remark in \cite{chafai} and they are related to the eigenvalue distribution of random matrices. In particular, the following relations hold for sufficiently smooth functions $\rho$
$$
W\ast\rho = -(-\partial_{xx})^{-(1+\alpha)/2}\rho + \frac{x^2}{2}\ast\rho\,, \quad\pa W\ast\rho = -\lt[\pa_x(-\partial_{xx})^{-(1+\alpha)/2}\rt]\rho+ x\ast \rho\,,
$$
and 
\eqv{\label{lappot}
-(-\partial_{xx})^{(1+\alpha)/2} (W\ast\rho) = \rho - (-\partial_{xx})^{\alpha/2}( x\ast \rho)\,.
}
Note that in the case $\alpha =0$, the derivative of $W\ast\rho$ is given by the Hilbert transform. The fractional operator $\pa_x(-\partial_{xx})^{-(1+\alpha)/2}$ when $-1<\alpha\leq 0$ has to be understood in the Cauchy principal value sense. 
With this information, we can now write the Euler-type equations for this potential in Lagrangian coordinates as
  \begin{subequations}\label{lag_eq_f}
    \begin{align}
      f(t,x)\frac{\pa \eta(t,x)}{\pa x} &= \rho_0(x), \quad (t,x) \in \R_+ \times \om_0, \label{lag_eqf_1}\\
    v'(t,x) + v(t,x) =&-\int_{\om(t)} \pa W(\eta(t,x) - y) \rho(t,y)dy \nonumber\\
      =& \int_{\om_0} \lt[\pa_x(-\partial_{xx})^{-(1+\alpha)/2} \lt(\frac{|x|^\alpha}{\alpha}\rt) \rt](\eta(t,x)- \eta(t,y))\rho_0(y)\,dy\nonumber\\
        &  - \int_{\om_0} (\eta(t,x) - \eta(t,y))\rho_0(y)\,dy\,. \label{lag_eqf_2}
    \end{align}
  \end{subequations}
Now, we would like to proceed by formally applying the differential operator $\pa_t^\alpha$ to \eqref{lag_eqf_2} taking into account \eqref{lappot} to find
$$\begin{aligned}
\pa_t^\alpha(v') + \pa_t^\alpha(v) 
=& \int_{\om_0} \delta(\eta(t,x)- \eta(t,y)) (v(t,x)- v(t,y))^\alpha\rho_0(y)\,dy\\ 
  &-\pa_t^{\alpha}(\eta)M_0  + \pa_t^{\alpha}\lt( \int_{\om_0} \eta(t,y)\rho_0(y)\,dy \rt)\\
=& -\pa_t^{\alpha-1}(v)M_0  + \pa_t^{\alpha - 1}\lt( \int_{\om_0} v(t,y)\rho_0(y)\,dy \rt)\,,
\end{aligned}$$
in case we are able to use the following chain rule for fractional derivatives 
$$
\pa_x^\alpha f(g(x)) = (\pa_g^\alpha f(g))\big|_{g = g(x)} \lt( \pa_x g(x)\rt)^\alpha.
$$
It is unclear though how to rigorously justify such chain rule, see \cite[Lemma 12]{Jum} for non-smooth settings. Assuming that $\pa_t^{\alpha - 1}$ is the inverse operator of $\pa_t^{1-\alpha}$, then we recover for $\alpha=1$ our core formula \eqref{eq_diff}. Using \eqref{mom} we can compute
$$
\pa_t^{\alpha - 1}\lt( \int_{\om_0} v(t,y)\rho_0(y)\,dy \rt) = \pa_t^{\alpha - 1}\lt( e^{-t}\int_{\om_0} (\rho_0 u_0)(y)\,dy  \rt) = M_1 \pa_t^{\alpha - 1}\lt(e^{-t}\rt)\,,
$$
by setting $w = \pa_t^{\alpha - 1}(v)$, we finally have
\bq\label{eq_w}
w'' + w' + M_0 w = M_1 \pa_t^{\alpha - 1}\lt(e^{-t}\rt).
\eq
Hence, we could try to solve the differential equation \eqref{eq_w} to get the explicit solution $w$. However, recovering $v$ and other quantities also needs a careful inversion of the involved fractional operators.
\end{remark}

%
%
%
%
\section{Blow-up phenomena of the system \eqref{main_eq} with pressure and viscosity}
In this section, we consider the barotropic compressible damped Navier-Stokes-Poisson equations with non-local interaction forces:
\begin{subequations}\label{eq2}
\begin{align}
&\pa_t \rho + \pa_x (\rho u) = 0, \qquad (t,x) \in \R_+ \times \om(t),\\
&\pa_t (\rho u) + \pa_x (\rho u^2) + \pa_x p(\rho) -\pa_x(\mu(\rho)\pa_x u)= -\rho u -(\pa W \star \rho)\rho,\label{eq2_2}
\end{align}
\end{subequations}
where $W(x) = - |x| + \frac{|x|^2}{2}$, subject to initial density and velocity
\bq\label{ini_eq2}
(\rho(t,\cdot)u(t,\cdot))|_{t=0} = (\rho_0,u_0).
\eq
Here the pressure law $p$ and the viscosity coefficient $\mu$ are given by $p(\rho) = \rho^\gamma$ and $\mu(\rho) = \rho^\alpha$ with $\gamma, \alpha > 1$.

Note that the term $\rho^{-1}\pa_x p$ is well-defined for the possible vacuum states $\rho = 0$ if $\gamma > 1$. We also notice that the pressure term in the system \eqref{eq2} can be formally derived from part of the potential term $\rho(\pa_x W \star \rho)$ by localizing part of $W$ near the origin. In this formal derivation, we obtain the system \eqref{eq2} with $\gamma = 2$. 

For the investigation of the finite-time blow-up, we assume that there exists a smooth $(\rho, u )\in \mc^2 \times \mc^3$ solutions in $\R \times [0,T^*)$ to the system \eqref{eq2} emanating from the initial data \eqref{ini_eq2} such that 
\bq\label{add-conti-p}
\lt(\sum_{0\leq k \leq 2}|\pa_x^k \rho_0(a_0)|\rt)\lt( \sum_{0\leq k \leq 2}|\pa_x^k \rho_0(b_0)| \rt)= 0.
\eq
By setting $d = \pa_x u$, we can easily verify that 
\begin{align}\label{re-eq2}
\begin{aligned}
& \dot\rho = - \rho d,\cr
& \dot d = -d^2 - d - \pa_x(\rho^{-1}\pa_xp(\rho)) +\pa_x(\rho^{-1}\pa_x(\mu(\rho)\pa_x u)) + 2\rho - M_0,
\end{aligned}
\end{align}
where $\dot{\xi}$ denotes the material derivative of $\xi$.
Then it follows from $\eqref{eq2_2}$ as in \cite[Lemma 2.1]{CH} that
\bq\label{est-press-1}
\sum_{0 \leq k \leq 2}|\pa_x^k \rho(t,\eta(t,x))| \leq \sum_{0\leq k \leq 2}|\pa_x^k \rho_0(x)|\exp\lt(C\sum_{1 \leq k \leq 3}\int_0^t |\pa_x^k u(s,\eta(s,x))|\,ds\rt).
\eq
We also notice that
\bq\label{est-press-2}
\pa_x(\rho^{-1}\pa_x \rho^\gamma) = \gamma(\gamma-2)\rho^{\gamma-3}(\pa_x\rho)^2 + \gamma \rho^{\gamma-2}\pa_x^2 \rho,
\eq
and
\begin{align}\label{est-press-3}
\begin{aligned}
\pa_x(\rho^{-1}\pa_x(\mu(\rho)\pa_x u))&=\rho^{\alpha-1}\pa_x^3 u+ 
(2\alpha-1)\rho^{\alpha-2}\pa_x\rho\pa_x^2 u
+\alpha\rho^{\alpha-2}\pa_x^2\rho\pa_x u\cr
&\quad +\alpha(\alpha-2)\rho^{\alpha-3}(\pa_x\rho)^2\pa_x u.
\end{aligned}
\end{align}
Thus the right hand sides of the equalities \eqref{est-press-2} and \eqref{est-press-3} are bounded if $\gamma, \alpha \in \{2\} \cup [3,\infty)$, and $\sum_{0\leq k \leq 2}|\pa_x^k \rho|$ and $\sum_{1\leq k \leq 3}|\pa_x^k u|$ are bounded. Taking into account \eqref{est-press-1} and \eqref{add-conti-p}, we deduce
\[
\lt(\sum_{0\leq k \leq 2}|\pa_x^k \rho_0(a(t))|\rt)\lt( \sum_{0\leq k \leq 2}|\pa_x^k \rho_0(b(t))| \rt)= 0 \quad \mbox{for} \quad t \in [0,T^*).
\]
Moreover it follows from \eqref{est-press-2} and \eqref{est-press-3} that
\[
\pa_x(\rho^{-1}\pa_x p)(t,y) = 0 \quad \mbox{and} \quad\pa_x((\rho^{-1}\pa_x(\mu(\rho)\pa_x u))(t,y) = 0,
\]
either for $y=a(t)$ or $y=b(t)$ for all $t \in [0,T^*)$. This implies from $\eqref{re-eq2}_2$ that
\bq\label{est_fin}
\mbox{either} \quad (\dot d +d^2 + d + M_0)(a(t)) = 0 \quad \mbox{or} \quad (\dot d +d^2 + d + M_0)(b(t)) = 0,
\eq
for all $t \in [0,T^*)$. 

\begin{theorem}\label{Th_NS}
Let $(\rho,u)$ be a $\mc^2 \times \mc^3$ classical solution in $\R\times [0,T^*)$ to the system \eqref{eq2}-\eqref{ini_eq2} with $\gamma, \alpha \in \{2\} \cup [3,\infty)$. Assume that either $x=a_0$ or $x=b_0$ satisfies 
\[
\sum_{0\leq k \leq 2}|\pa_x^k \rho_0(x)| = 0   \qquad \mbox{and} \qquad d_0(x) :=d(0)< d_- = \frac{-1 - \sqrt{1 - 4M_0}}{2}\,.
\]
Then  $T^*$ is finite. Furthermore, we have
\[
T^* \leq \min_{x \in \{a_0,b_0\}} \frac{1}{d_- - d_0(x)}.
\]
\end{theorem}
\begin{proof}It follows from \eqref{est_fin} that for $1 - 4M_0 > 0$
\[
\dot d = -(d^2 + d + M_0) = -(d - d_+)(d-d_-), \quad \mbox{where} \quad d_\pm = \frac{-1 \pm \sqrt{1 - 4M_0}}{2}.
\]
If $d_0 < d_-$, then 
\[
\dot d \leq -(d - d_-)^2 \quad \mbox{and} \quad d \leq \frac{d_0 - d_-}{1 + (d_0 - d_-)t} + d_-.
\]
Since $d_0 - d_- < 0$, thus $d(t,y) = \pa_x u(t,y)$ with $y=a(t)$ or $y=b(t)$ will blow up before the time $T^*$ which satisfies 
\[
T^* \leq \min_{x \in \{a_0,b_0\}} \frac{1}{d_- - d_0(x)}.
\]
This completes the proof.

\end{proof}
\begin{remark}
Theorem {\rm\ref{Th_NS}} can be generalized to the case of compactly supported initial density with possible vacuum regions $\rho_0=0$. 
\end{remark}

%
%
%
%
\appendix
\section{Existence and uniqueness of local-in-time classical solutions}
In this section, we study the existence of local-in-time classical solutions to the system \eqref{lag_eq}. We prove the following theorem
\begin{theorem}\label{thm_local} Let $s \geq 1$. Suppose that $(\rho_0,u_0) \in H^s(\om_0) \times H^{s+1}(\om_0)$. Then for any constants $0<M < \widetilde{M}$ there exists a $T_0 > 0$,  depending only on $M$ and $\widetilde M$, such that if $\|u_0\|_{H^{s+1}} < M$, then the system \eqref{lag_eq} has a unique solution $(f,v) \in \mc([0,T_0]; H^s(\om_0)) \times \mc([0,T_0]; H^{s+1}(\om_0))$ satisfying
\[
\sup_{0 \leq t \leq T_0} \|v(t,\cdot)\|_{H^{s+1}} \leq \widetilde M.
\]
\end{theorem}
\begin{proof}
We approximate the solutions of system  \eqref{lag_eq} by the sequence $\eta^n,v^n$ solving the integro-differential system:
\begin{subequations}\label{eq_app}
\begin{align}
\pa_t \eta^{n+1}(t,x) &= v^n(t,x), \quad x \in \om_0, \quad t > 0,\label{eq_app1}\\
\pa_t v^{n+1}(t,x) &= -v^{n+1}(t,x) - \int_{\om_0} \pa W(\eta^{n+1}(t,x) - \eta^{n+1}(t,y)) \rho_0(y)\,dy, \label{eq_app2}
\end{align}
\end{subequations}
with the initial data and first iteration step defined by
\[
(\eta^n(t,x),v^n(t,x))|_{t=0} = (x,u_0) \qquad \mbox{for all} \quad n \geq 1, \quad  x \in \om_0,
\]
and
\[
v^0(t,x) = u_0, \quad (t,x) \in \R_+ \times \om_0.
\]
To simplify the notation, from now on we drop the dependence on the spatial domain in the symbols of functional spaces.

$\bullet$ {\bf Step 1}. (Uniform bounds): We claim that there exists $T_0 > 0$ such that
\[
\sup_{0 \leq t \leq T_0} \|v^n(t,\cdot)\|_{H^{s+1}} \leq \widetilde M \quad \mbox{for} \quad n \in \N \cup \{0\}.
\]
To prove this claim, we use an induction argument. In the first iteration step, we find that
\[
\sup_{0 \leq t \leq T} \|v^0(t,\cdot)\|_{H^{s+1}} = \|u_0\|_{H^{s+1}} \leq M < \widetilde M.
\]
Let us assume that 
\[
v^n \in \mc([0,T];H^{s}) \quad \mbox{and} \quad \sup_{0 \leq t \leq T} \|v^n(t,\cdot)\|_{H^{s+1}} \leq \widetilde M,
\]
for some $T > 0$. Then we check that the linear approximations $(\eta^{n+1},v^{n+1})$ from the system \eqref{eq_app} are well-defined and they satisfy  $(\eta^{n+1},v^{n+1}) \in \mc([0,T];H^{s+1}) \times \mc([0,T];H^{s+1})$. We begin by estimating $\eta^{n+1}$. It follows from \eqref{eq_app1} that
\[
\eta^{n+1}(t,x) = x + \int_0^t v^n(s,x)\,ds \quad \mbox{and} \quad \pa_x^k \eta^{n+1}(t,x) = \delta_{k,1} + \int_0^t \pa_x^k v^n(s,x)\,ds,
\]
for $k \geq 1$, where $\delta_{k,1}$ denotes Kronecker delta, i.e., $\delta_{k,1} = 1$ if $k=1$ and $\delta_{k,1} = 0$ otherwise. From this expression, it is straightforward to get
\[
\|\eta^{n+1}(t,\cdot)\|_{L^2} \leq C|\om_0| + \int_0^t \|v^n(s,\cdot)\|_{L^2}\,ds \leq C|\om_0| + T\|v^n\|_{L^\infty(0,T;L^2)}
\]
and
\[
\|\eta^{n+1}(t,\cdot)\|_{\dot{H}^k} \leq \sqrt{|\om_0|}\delta_{k,1} + \int_0^t \|v^n(s,\cdot)\|_{\dot{H}^k}ds \leq \sqrt{|\om_0|}\delta_{k,1} + T\|v^n(s,\cdot)\|_{\dot{H}^k},
\]
for some $T > 0$, where $\dot{H}^k$ represents the homogeneous Sobolev space. This yields
\[
\sup_{0 \leq t \leq T}\|\eta^{n+1}(t,\cdot)\|_{H^s} \leq C(\om_0) + T \widetilde{M} =: C_0.
\]
Moreover, we find that there exists $T_1$,   such that  $0< T\leq T$ and
\[
\pa_x \eta^{n+1}(t,x) = 1 + \int_0^t \pa_x v^n(s,x)\,ds \geq 1 - T_1 \widetilde M > 0.
\]
For the estimate of $\|v^n\|_{H^{s+1}}$, we first notice that
\[
sgn(\eta^{n+1}(t,x) - \eta^{n+1}(t,y)) = sgn(x-y) \quad \mbox{for} \quad t \in [0,T],
\]
since $\eta^{n+1}(t,x)$ is uniquely well-defined, i.e., there are no crossing between trajectories. This enables us to rewrite \eqref{eq_app2} as
\[
\pa_t v^{n+1}(t,x) = -v^{n+1}(t,x) + \int_{\om_0} sgn(x - y) \rho_0(y)\,dy - \eta^{n+1}(t,x)M_0 + \int_{\om_0} \eta^{n+1}(t,y)\rho_0(y)\,dy,
\]
and further, solving the above ODE we get
\begin{align}\label{exp_vn}
\begin{aligned}
v^{n+1}(t,x) &= u_0(x)e^{-t} + (1 - e^{-t})\lt(2\int_{a_0}^x \rho_0(y)\,dy - M_0 \rt) - M_0\int_0^t e^{-(t-s)}\eta^{n+1}(s,x)\,ds\cr
&\quad + \int_0^t \int_{\om_0} e^{-(t-s)}\eta^{n+1}(s,y)\rho_0(y)\,dy\,ds.
\end{aligned}
\end{align}
For the spatial-derivative, we easily find
\bq\label{exp_vn1}
\pa_x^k v^{n+1}(t,x) = \pa_x^k u_0(x) e^{-t} + 2(1 - e^{-t})\pa_x^{k-1} \rho_0(x) -M_0 \int_0^t e^{-(t-s)} \pa_x^k \eta^{n+1}(s,x)\,ds,
\eq
for $k\geq1$. Then, we obtain from \eqref{exp_vn} and \eqref{exp_vn1} that
$$\begin{aligned}
\|v^{n+1}(t,\cdot)\|_{L^2} 
& \leq e^{-t}\|u_0\|_{L^2} + M_0(1-e^{-t})\sqrt{|\om_0|} + M_0\int_0^t e^{-(t-s)}\|\eta^{n+1}(s,\cdot)\|_{L^2}\,ds\cr
&\quad + \sqrt{|\om_0|} \int_0^t \int_{\om_0} e^{-(t-s)} \eta^{n+1}(s,y)\rho_0(y)\,dy\,ds\cr
&\leq e^{-t}\|u_0\|_{L^2} + \left(\sqrt{|\om_0|}M_0 + M_0 C_1 +C_1 |\om_0| \|\rho_0\|_{L^\infty} \right)  (1 - e^{-t})\cr
& = e^{-t}\|u_0\|_{L^2} + (M_0(C_1 + \sqrt{|\om_0|}) +C_1 |\om_0| \|\rho_0\|_{L^\infty})(1 - e^{-t})
\end{aligned}$$
and
$$\begin{aligned}
\|v^{n+1}(t,\cdot)\|_{\dot{H}^k} &\leq e^{-t}\|u_0\|_{\dot{H}^k} + \left( 2 \|\rho_0\|_{\dot{H}^{k-1}} + M_0C_1\right) (1 - e^{-t}) \quad \mbox{for} \quad k \geq 1,
\end{aligned}$$
respectively. Thus we conclude
\bq\label{est_vn}
\|v^{n+1}(t,\cdot)\|_{H^{s+1}} \leq e^{-t}\|u_0\|_{H^{s+1}} + C_2(1 - e^{-t}),
\eq
where $C_2 > 0$ is given by
\[
C_2 := M_0(C_1 + \sqrt{|\om_0|}) + \|\rho_0\|_{L^\infty} C_1 |\om_0| + 2\|\rho_0\|_{H^s} + M_0 C_1.
\]
The r.h.s. of \eqref{est_vn}:
\[
h(t) := e^{-t}\|u_0\|_{H^{s+1}} + C_2(1 - e^{-t}),
\]
is a decreasing function of time and $h(0) = \|u_0\|_{H^{s+1}} < M < \widetilde M$. This implies that we can choose $T_0$ small enough such that $0< T_0 \leq T_1$ and
\[
\sup_{0 \leq t \leq T_0} \|v^{n+1}(t,\cdot)\|_{H^{s+1}} \leq \widetilde M.
\]

$\bullet$ {\bf Step 2}. (Cauchy estimates): Set
\[
\etan(t,x) := \eta^{n+1}(t,x) - \eta^n(t,x) \quad \mbox{and} \quad \vn(t,x) := v^{n+1}(t,x) - v^n(t,x).
\]
Then we find that $\etan$ and $\vn$ satisfy
\[
\etan(t,x) = \int_0^t v^{n,n-1}(s,x)\,ds
\]
and
\[
\vn(t,x) = -M_0\int_0^t e^{-(t-s)} \etan(s,x)\,ds + \int_0^t \int_{\om_0} e^{-(t-s)}\etan(s,y)\rho_0(y)\,dy\,ds.
\]
This yields
\[
\|\etan(t,\cdot)\|_{L^2} \leq \int_0^t \|v^{n,n-1}(s,\cdot)\|_{L^2}\,ds
\]
and
$$\begin{aligned}
\|\vn(t,\cdot)\|_{L^2} &\leq  M_0\int_0^t\!\! e^{-(t-s)}\|\etan(s,\cdot)\|_{L^2}\,ds 
 + \|\rho_0\|_{L^\infty}\sqrt{|\om_0|}\int_0^t\!\! e^{-(t-s)}\|\etan(s,\cdot)\|_{L^2}\,ds\cr
&\leq (M_0 + \|\rho_0\|_{L^\infty}\sqrt{|\om_0|})\int_0^t \|\etan(s,\cdot)\|_{L^2}\,ds.
\end{aligned}$$
Introducing $\Delta_{\eta,v}^{n+1}(t):= \|\etan(t,\cdot)\|_{L^2} + \|\vn(t,\cdot)\|_{L^2}$ and combining the above estimates, we get
\[
\Delta_{\eta,v}^{n+1}(t) \leq C\int_0^t \Delta_{\eta,v}^n(s)\,ds \quad \mbox{for some} \quad C > 0.
\]
This implies
\[
\|\etan(t,\cdot)\|_{L^2} + \|\vn(t,\cdot)\|_{L^2} \ls \frac{T_0^{n+1}}{(n+1)!},
\]
for $t \leq T_0$. Thus, we find that $(\eta^n(t,x), v^n(t,x))$ is a Cauchy sequence in $\mc([0,T_0]; L^2) \times \mc([0,T_0]; L^2)$.

$\bullet$ {\bf Step 3}. (Regularity of limiting functions): It follows from  Step 2 that there exist limit functions $\eta$ and $v$ such that 
\[
(\eta^n, v^n) \to (\eta,v) \quad \mbox{in} \quad \mc([0,T_0]; L^2) \times \mc([0,T_0]; L^2).
\]
Interpolating this with the uniform bound estimates in Step 1, we obtain
\bq\label{est_lim1}
(\eta^n, v^n) \to (\eta,v) \quad \mbox{in} \quad \mc([0,T_0]; H^s) \times \mc([0,T_0]; H^s) \quad \mbox{as} \quad n \to \infty.
\eq
We now claim that $(\eta,v) \in \mc([0,T_0]; H^{s+1}) \times \mc([0,T_0]; H^{s+1})$. Note that we can easily check that $v \in \mc([0,T_0]; H^{s+1})$ implies $\eta \in \mc^1([0,T_0]; H^{s+1})$ due to the above convergence and \eqref{eq_app1}. Thus, it suffices to show that  $v \in \mc([0,T_0]; H^{s+1})$.  It follows from  Step 1 that there exists a weakly convergent subsequence $v^{n_k} \rightharpoonup \tilde v$ as $k \to \infty$, such that
\[
\|\tilde v(t,\cdot)\|_{H^{s+1}} \leq \liminf_{k \to \infty}\|v^{n_k}(t,\cdot)\|_{H^{s+1}}, \quad t \in [0,T_0],
\]
for some $\tilde v \in L^\infty(0,T;H^{s+1})$. This together with \eqref{est_lim1} yields
\[
\tilde v(t) = v(t) \quad \mbox{in} \quad H^{s+1} \quad \mbox{for each} \quad t \in [0,T_0].
\]
Thus we have
\[
\sup_{0 \leq t \leq T_0} \|v(t,\cdot)\|_{H^{s+1}} \leq \widetilde M.
\]
We next show that 
\bq\label{est_reg0}
v \in \mc_w([0,T_0]; H^{s+1}), \quad \mbox{i.e.,}  \quad v(t) \rightharpoonup v(t_0) \mbox{ in } H^{s+1} \quad \mbox{as} \quad t \to t_0, 
\eq 
for $t_0 \in [0,T_0].$ Without loss of generality, we may assume $t_0 = 0$. Then we obtain from the weak lower semi-continuity and \eqref{est_lim1} that
\bq\label{est_reg1}
\|u_0\|_{H^{s+1}} \leq \liminf_{t \to 0+}\|v(t)\|_{H^{s+1}}.
\eq
Thus the weak continuity can be obtained from the strong convergence \eqref{est_lim1} and \eqref{est_reg1}. Indeed, for a sequence $t_k \subset [0,T]$ such that $t_k \to 0$ as $k \to \infty$, we have $\lim_{k \to \infty}\|v(t_k) - u_0\|_{H^s} \to 0$ due to \eqref{est_lim1} and $\|u_0\|_{H^{s+1}} \leq \widetilde M$. 

On the other hand, it follows from \eqref{est_vn} and the weak lower semi-continuity that
\bq\label{est_reg2}
\limsup_{t \to 0+}\|v(t,\cdot)\|_{H^{s+1}} \leq \|u_0\|_{H^{s+1}}.
\eq
Combining \eqref{est_reg1} and \eqref{est_reg2}, we find 
\[
\|v(t,\cdot)\|_{H^{s+1}} \to \|v(t_0,\cdot)\|_{H^{s+1}}, 
\]
as $t \to t_0+$, and this together with \eqref{est_reg0} implies 
\[
\lim_{t \to t_0+}\|v(t) - v_0\|_{H^{s+1}} = 0 \quad \mbox{for} \quad t_0 \in [0,T_0].
\]
For the continuity from the left hand side, we use a change of variable $t \mapsto T_0 - t$ by taking into account the time-reversed problem. 

$\bullet$ {\bf Step 4}. (Existence): In Step 3, we found
\[
(\eta^n, v^n) \to (\eta, v) \quad \mbox{in } \mc([0,T_0]; H^s),
\]
and this implies that the limit functions $(\eta,v)$ are solutions to \eqref{eq_tra}-\eqref{lag_eq2} in the sense of distributions. 
In Step 3, we also proved 
that $(\eta,v) \in \mc([0,T_0]; H^{s+1})$ and
\[
\sup_{0 \leq t \leq T_0} \|v(t,\cdot)\|_{H^{s+1}} \leq \widetilde M.
\]
Subsequently, we get
\[
\inf_{0 \leq t \leq T_0} \inf_{x \in \om_0} \pa_x \eta(t,x) > 0.
\]
Finally, we use the expression for $f$ in \eqref{lag_eq1} together with the above estimate of $\pa_x \eta$ to deduce $f \in \mc([0,T_0]; H^s)$.

$\bullet$ {\bf Step 5}. (Uniqueness): Let $(f,v)$ and $(\tilde f, \tilde v)$ be the two classical solutions constructed in the previous steps corresponding to the same initial data $(\rho_0,u_0)$. Set $\eta$ and $\tilde \eta$ the trajectories with respect to $v$ and $\tilde v$, respectively, i.e.,
\[
\pa_t \eta(t,x) = v(t,x) \quad \mbox{and} \quad \pa_t \tilde \eta(t,x) = \tilde v(t,x) \quad \mbox{for} \quad (t,x) \in [0,T_0] \times \om_0.
\]
Then similarly as in Step 2 we get
\[
\|v(t,\cdot) - \tilde v(t,\cdot)\|_{L^2} \leq C\int_0^t \|v(s,\cdot) - \tilde v(s,\cdot)\|_{L^2}.
\]
This yields
\[
v \equiv \tilde v \quad \mbox{in} \quad \mc([0,T_0];L^2).
\]
Furthermore, we can easily check $v = \tilde v$ in $\mc([0,T_0];H^{s+1})$ by using the similar argument as before, in Step 3. In particular, this concludes 
\[
\pa_x \eta(t,x) = \pa_x \tilde \eta(t,x) \quad \mbox{in} \quad \mc^1([0,T_0];H^s).
\]
Hence, we obtain
\[
f(t,x) = \frac{\rho_0(x)}{\pa_x \eta(t,x)} = \frac{\rho_0(x)}{\pa_x \tilde \eta(t,x)} = \tilde f(t,x)\quad \mbox{in} \quad \mc([0,T_0];H^s).
\]
\end{proof}

\begin{remark}
It follows from Theorem {\rm\ref{thm_local}} that 
\[
(f,v) \in \lt(\mc^1([0,T]; H^{s-1}) \cap \mc([0,T_0]; H^s)\rt) \times \lt(\mc^1([0,T]; H^s) \cap \mc([0,T_0]; H^{s+1})\rt)
\]
for $s \geq 1$, due to the structure of the system \eqref{lag_eq}. In particular, if $s = 2$, then we have
\[
(f,v) \in \lt(\mc^1([0,T]; \mc(\om_0)) \cap \mc([0,T_0]; \mc^1(\om_0))\rt) \times \lt(\mc^1([0,T]; \mc^1(\om_0)) \cap \mc([0,T_0]; \mc^2(\om_0))\rt).
\]
Using the regularity for $v$, we get
\bq\label{reg_eta}
\eta' = v \in \mc([0,T_0]; \mc^2(\om_0)), \quad \mbox{i.e.,} \quad \eta \in \mc^1([0,T_0]; \mc^2(\om_0)).
\eq
On the other hand, by expanding the interaction term in \eqref{lag_eq2}, we find
\bq\label{eq_exp}
v' = -v - M_0\eta + \int_{\om_0} x \rho_0(x)\,dx + (1 - e^{-t})\int_{\om_0} \rho_0(x) u_0(x)\,dx,
\eq
due to 
\[
\frac{d}{dt} \int_{\om_0} \eta(t,x)\rho_0(x)\,dx = e^{-t}\int_{\om_0} \rho_0(x)u_0(x)\,dx.
\]
This and together with the regularity for $\eta$ and $v$ yields
\[
v \in \mc^1([0,T_0]; \mc^2(\om_0)),
\]
and again use the above regularity, \eqref{reg_eta}, and \eqref{eq_exp} to obtain
\[
v \in \mc^2([0,T_0] \times \om_0).
\]
For the regularity of $f$, we easily find that
\[
f = \frac{\rho_0}{\pa_x \eta} \in \mc^1([0,T_0] \times \om_0),
\]
since
\[
\rho_0 \in \mc^1(\om_0) \quad \mbox{and} \quad \pa_x \eta \in \mc^1([0,T_0] \times \om_0).
\]
Hence if we assume $(\rho_0,u_0) \in H^2(\om_0) \times H^3(\om_0)$, we have the unique local-in-time $\mc^1 \times \mc^2$-solution $(f,v)$ for the system \eqref{lag_eq}.
\end{remark}

By using this local-in-time existence and uniqueness results, it is trivial to construct classical solutions to the problem \eqref{main_eq} in the sense given in the introduction up to a maximal time interval $T>0$ by the standard procedure of continuing the solutions as long as the bounds are satisfied.


\section*{Acknowledgements}
JAC was partially supported by the Royal Society via a Wolfson Research Merit Award. YPC was supported by the ERC-Starting grant HDSPCONTR ``High-Dimensional Sparse Optimal Control''. JAC and YPC were partially supported by EPSRC grant EP/K008404/1. EZ has been partly supported by the National  Science  Centre  grant 2014/14/M/ST1/00108 (Harmonia).
The authors warmly thank Sergio P\'erez for providing us with the numerical results included in Section 4. We also thank the department of Mathematics at KAUST, and particularly A. Tzavaras, for their hospitality during part of this work.

%
%
%
%

\end{document}